\documentclass[12pt]{article}
\usepackage{amsmath}
\usepackage{amssymb}
\usepackage{amsfonts}
\usepackage{mitpress}

\newtheorem{theorem}{Theorem}
\newtheorem{acknowledgement}[theorem]{Acknowledgement}

\newtheorem{lemma}[theorem]{Lemma}

\newtheorem{problem}[theorem]{Problem}

\newenvironment{proof}[1][Proof]{\noindent\textbf{#1.} }{\ \rule{0.5em}{0.5em}}
\input{tcilatex}

\begin{document}

\title{Splitting multidimensional necklaces and measurable colorings of
Euclidean spaces}
\author{Jaros\l aw Grytczuk, Wojciech Lubawski \\
Theoretical Computer Science Department, Faculty of Mathematics and Computer
Science, Jagiellonian University, 30-348 Krak\'{o}w, Poland;
grytczuk@tcs.uj.edu.pl, lubawski@tcs.uj.edu.pl}
\maketitle

\begin{abstract}
A necklace splitting theorem of Goldberg and West \cite{GoldbergWest}
asserts that any $k$-colored (continuous) necklace can be fairly split using
at most $k$ cuts. Motivated by the problem of Erd\H{o}s on strongly
nonrepetitive sequences, Alon et al. \cite{AlonGLM} proved that there is a $%
(t+3)$-coloring of the real line in which no necklace has a fair splitting
using at most $t$ cuts. We generalize this result for higher dimensional
spaces. More specifically, we prove that there is $k$-coloring of $\mathbb{R}%
^{d}$ such that no cube has a fair splitting of size $t$ (using at most $t$
hyperplanes orthogonal to each of the axes), provided $k\geq
(t+4)^{d}-(t+3)^{d}+(t+2)^{d}-2^{d}+d(t+2)+3$. We also consider a discrete
variant of the multidimensional necklace splitting problem in the spirit of
the theorem of de Longueville and \v{Z}ivaljevi\'{c} \cite{LonguevilleZ}.
The question how many axes aligned hyperplanes are needed for a fair
splitting of a $d$-dimensional $k$-colored cube remains open.
\end{abstract}

\section{Introduction}

In this paper we investigate some questions connected to the necklace
splitting problem. Let $c:\mathbb{R}\rightarrow \{1,2,\ldots ,k\}$ be a $k$%
-coloring of the real line. We assume that $c$ is a\emph{\ measurable}
coloring, that is, the set $c^{-1}(i)$ of all points in color $i$ is
Lebesgue measurable for every $i\in \{1,2,\ldots ,k\}$. A \emph{splitting}
of size $t$ of an interval $[a,b]$ is a sequence of points $%
a=y_{0}<y_{1}<\ldots <y_{t}<y_{t+1}=b$. A splitting is said to be \emph{fair}
if it is possible to partition the resulting collection of intervals $%
F=\{[y_{i},y_{i+1}]:0\leq i\leq t\}$ into two disjoint subcollections $F_{1}$
and $F_{2}$, each capturing exactly half of the total measure of every
color. The partition $F=F_{1}\cup F_{2}$ will be called a \emph{fair}
partition of $F$.

Goldberg and West \cite{GoldbergWest} proved that every $k$-colored interval
has a splitting of size at most $k$ (see also \cite{AlonWest} for a short
proof using the Borsuk-Ulam theorem, and \cite{Matousek} for other
applications of the Borsuk-Ulam theorem in combinatorics). This result is
clearly the best possible, as can be seen in a necklace where colors occupy
consecutively full intervals.

In \cite{AlonGLM} we considered colorings of $\mathbb{R}$ such that no
interval has a splitting of bounded size.

\begin{theorem}
\label{AGLM}\emph{(Alon et al. \cite{AlonGLM})} For every $t\geq 1$ there is
a $(t+3)$-coloring of the real line such that no interval has a fair
splitting of size at most $t$.
\end{theorem}

For $t=1$ the result asserts that there is a $4$-coloring of the real line
avoiding (continuous) \emph{abelian squares} (adjacent intervals with equal
measure of every color). The question whether a similar property holds for
the integers was posed in 1961 by Erd\H{o}s \cite{Erdos}, and solved in the
affirmative by Ker\"{a}nen \cite{Keranen}\ in 1991. Curiously the number of
colors is the same in both versions, though in continuous variant it is not
known whether it is optimal.

In this paper we extend the above result in the spirit of the theorem of de
Longueville and \v{Z}ivaljevi\'{c} \cite{LonguevilleZ}. Let $d$ be a fixed
positive integer, and let $c$ be a measurable coloring of $\mathbb{R}^{d}$. A%
\emph{\ cube} in $\mathbb{R}^{d}$ is just a Cartesian product of $d$
non-empty intervals (of the same length) lying on distinct coordinate axes.
A \emph{splitting} of a cube is specified by a family of axes-aligned
hyperplanes. A splitting of a colored cube is \emph{fair} if there is a
partition of the resulting family of cuboids into two families, each
capturing exactly half of the total measure of every color.

\begin{theorem}
\label{deLZ}\emph{(de Longueville and \v{Z}ivaljevi\'{c} \cite{LonguevilleZ})%
} Every $k$-colored\emph{\ }$d$-dimensional cube has a fair splitting using
at most $k$ axes aligned hyperplane cuts. Moreover, one may specify the
number of cuts in each direction arbitrarily.
\end{theorem}

We are interested in colorings avoiding cubes admitting a fair splitting
with a bounded number of cuts. The \emph{size} of the splitting is the
maximum number of axes aligned hyperplanes in the same direction. Our main
result reads as follows.

\begin{theorem}
\label{Main} For every pair of integers $t$,$d\geq 1$, and $k\geq
(t+4)^{d}-(t+3)^{d}+(t+2)^{d}-2^{d}+d(t+2)+3$, there is a $k$-coloring of $%
\mathbb{R}^{d}$ such that no cube has a fair splitting of size at most $t$.
\end{theorem}

The proof uses Baire category argument applied to the space of all
measurable colorings of $\mathbb{R}^{d}$. The lower bound on the number of
colors in the above theorem is almost surely not optimal. In the final
section we discuss several open problems and further directions. The most
intriguing seems a discrete version of the multidimensional necklace
splitting problem: what is the least number of axes aligned hyperplanes
needed to a fair splitting of a discrete $k$-colored cuboid in $\mathbb{R}%
^{d}$?

\section{Proof of the main result}

Recall that a set in a metric space is \emph{nowhere dense} if the interior
of its closure is empty. A set is said to be of \emph{first category} if it
can be represented as a countable union of nowhere dense sets. In the proof
of theorem \ref{Main} we apply the Baire category theorem (see \cite{Oxtoby}%
).

\begin{theorem}
\emph{(Baire Category Theorem)} If $X$ is a complete metric space and $A$ is
a set of first category in $X$, then $X\setminus A$ is dense in $X$ (and in
particular is nonempty).
\end{theorem}

Our plan is to follow a similar reasoning to that of \cite{AlonGLM}. We will
construct a suitable metric space of colorings of $\mathbb{R}^{d}$, and then
we will demonstrate that the subset of \textquotedblleft bad
colorings\textquotedblright\ is of first category.

\subsection{The setting}

Let $k$ be a fixed positive integer and let $\{1,2,\ldots ,k\}$ be the set
of colors. Let $f$ and $g$ be two measurable colorings of $\mathbb{R}^{d}$.
For a positive integer $n$ we set 
\begin{equation*}
D_{n}(f,g)=\{x\in \lbrack -n,n]^{d}\colon f(x)\neq g(x)\}.
\end{equation*}%
Clearly $D_{n}(f,g)$ is Lebesgue measurable so we may define the normalized
distance between $f$ and $g$ on $[-n,n]^{d}$ by 
\begin{equation*}
d_{n}(f,g)=\frac{\lambda (D_{n}(f,g))}{n^{d}},
\end{equation*}%
where $\lambda $ is the $d$-dimensional Lebesgue measure. Since $d_{n}(f,g)$
is bounded from above by $2^{d}$, we may define the distance between two
measurable colorings $f$ and $g$ by 
\begin{equation*}
d(f,g)=\sum_{n=1}^{\infty }\frac{d_{n}(f,g)}{2^{n+1}}.
\end{equation*}%
Identifying colorings whose distance is zero gives a metric space $\mathcal{M%
}$ of equivalence classes of all measurable $k$-colorings. Note that the
splitting properties are preserved by equivalent colorings.

\begin{lemma}
The space $\mathcal{M}$ is a complete metric.
\end{lemma}

We omit the proof of this lemma since this is a simple generalization of a
result stating that sets of finite measure in any metric space form a
complete metric space with symmetric difference as the distance function
(see \cite{AlonGLM}, \cite{Oxtoby}).

Let $t\geqslant 1$ be a fixed integer. Let $D_{t}$ be a subspace of $%
\mathcal{M}$ consisting of those $k$-colorings that avoid intervals having a 
$d$-dimensional fair splitting of size at most $t$ in each dimension. Denote
for future convenience%
\begin{equation*}
f(d,t)=(t+4)^{d}-(t+3)^{d}+(t+2)^{d}-2^{d}+d(t+2)+3.
\end{equation*}%
We will show that $D_{t}$ is not empty provided that $k\geqslant f(d,t)$. By 
\emph{granularity} of a splitting we mean the length of the shortest
subinterval $[z_{j}^{i},z_{j+1}^{i}]$ in the splitting. For $n\geqslant 1$
and $r_{1},\ldots ,r_{n}$, let $B_{n}^{(r_{i})}$ be the set of those
colorings from $\mathcal{M}$ for which there exists at least one $d$%
-dimensional cube in $[-n,n]^{d}$ having a $d$-dimensional fair splitting of
size exactly $r_{i}$ in the $i$-th dimension for each $i$ and granularity at
least $1/n$. Finally let us denote all the bad colorings by 
\begin{equation*}
B_{n}(t)=\bigcup_{r_{i}\leqslant t}B_{n}^{(r_{i})}.
\end{equation*}%
Clearly we have 
\begin{equation*}
D_{t}=\mathcal{M}\setminus \bigcup_{n=1}^{\infty }B_{n}(t).
\end{equation*}%
Now our aim is to apply Baire category theorem to show that the sets $%
B_{n}(t)$ are nowhere dense, provided that $k\geqslant
(t+4)^{d}-(t+3)^{d}+(t+2)^{d}-2^{d}+d(t+2)+3$.

\subsection{The sets $B_{n}(t)$}

We show that each set $B^{(r_i)}_n$ is a closed subset of $\mathcal{M}$.
Since $B_n(t)$ is a finite union of these sets, it must be closed too.

\begin{theorem}
The set $B_{n}^{(r_{i})}$ is a closed subset of $\mathcal{M}$ for every $%
r_{i}\geqslant 1$ and $n\geqslant 1$.
\end{theorem}

\begin{proof}
Let $\{f_{m}\}$ be a sequence of colorings converging in $\mathcal{M}$ to $f$%
. For each $m$ let $C_{m}$ denote a $d$-dimensional cube in $[-n,n]^{d}$ of
granularity $\geqslant 1/n$ and having a fair splitting into exactly $r_{i}$
points in the $i$-th dimension. Let us denote by $\phi _{m}\colon \lbrack
r_{1}]\times \ldots \times \lbrack r_{d}]\rightarrow \{1,2\}$ the labeling
function defining the two families from the fair splitting of $C_{m}$. Since 
$[-n,n]^{d}$ is compact we may assume that vertices of the sliced cube $%
C_{m} $ converge to vertices of some cube $C$ and since there is finite
number of labeling functions we may assume that $\phi _{m}=\phi $ for every $%
m$. Now it is easy to see that $\phi $ gives a fair splitting for $C$.
\end{proof}

Next we prove that each $B_{n}(t)$ has empty interior provided the number of
colors $k$ satisfies $k>(t+4)^{d}-(t+3)^{d}+(t+2)^{d}-2^{d}+d(t+2)+2$. For
this purpose let us call $f\in \mathcal{M}$ a cube coloring on $[-n,n]^{d}$
if there is a partition of $[-n,n]^{d}$ into some number of (half open) $d$%
-dimensional cubes of equal size in each dimension, each filled with only
one color. Let $I_{n}$ denote the set of all colorings from $\mathcal{M}$
that are cube colorings on $[-n,n]^{d}$.

\begin{lemma}
\label{lem.3} Let $f\in \mathcal{M}$ be a $k$-coloring. Then for every $%
\epsilon >0$ and $n\in \mathbb{N}$ there exists a coloring $g\in I_{n}$ such
that $d(f,g)<\epsilon $.
\end{lemma}

\begin{proof}
Let $C_{i}=f^{-1}(i)\cap \lbrack -n,n]^{d}$ and let $C_{i}^{\ast }\subseteq
\lbrack -n,n]^{d}$ be a finite union of intervals such that 
\begin{equation*}
\lambda \left( (C_{i}^{\ast }\backslash C_{i})\cup (C_{i}\backslash
C_{i}^{\ast })\right) <\frac{\epsilon }{2k^{2}}
\end{equation*}%
for each $i=1,2,\ldots ,k$. Define coloring $h$ so that for each $%
i=1,2,\ldots ,k$ the set $C_{i}^{\ast }\backslash (C_{1}^{\ast }\cup \ldots
\cup C_{i-1}^{\ast })$ is filled with color $i$, the rest of the cube $%
[-n,n]^{d}$ is filled with any of these colors. Moreover we set $h$ to be
equal $f$ outside $[-n,n]^{d}$. Note that $d(f,h)<\epsilon /2$ and $%
h^{-1}(i)\cap \lbrack -n,n]^{d}$ is a finite union of cubes. Let $%
A_{1},A_{2},\ldots ,A_{N}$ be the whole family of these cubes. Now split the
cube $[-n,n]^{d}$ into $M^{d}$ cubes $B_{1},\ldots B_{M^{d}}$ equally spaced
in $[-n,n]^{d}$. We define $g$ to be equal $h$ on $A_{i}$ whenever $%
A_{i}\subseteq B_{j}$ for some $j$ and $g(A_{i})$ is of any color otherwise.
Note that $g$ differs from $h$ on a set of $d$-dimensional measure at most $%
t((2n+4n/M)^{d}-(2n)^{d})$ so that for sufficiently large $M$ $%
d(g,h)<\epsilon /2$ and we get $d(f,g)<\epsilon $.
\end{proof}

In order to state the next lemma we will use the following notation: 
\begin{equation*}
D(d)=\sum_{i=1}^{d}{\binom{d}{i}}%
(t+2)^{i}(2^{d-i}-1)=(t+4)^{d}+1-(t+3)^{d}-2^{d}.
\end{equation*}

\begin{lemma}
If $k>(t+2)^{d}+d(t+2)+1+D(d)$ then each $B_{n}(t)$ has empty interior.
\end{lemma}

\begin{proof}
Let $f\in B_{n}(t)$ be any bad coloring. Let $U(f,\epsilon )$ be the open $%
\epsilon $-neighborhood of $f$ in the space $\mathcal{M}$. Assume the
assertion of the lemma is false: there is some $\epsilon >0$ for witch $%
U(f,\epsilon )\subseteq B_{n}(t)$. By Lemma \ref{lem.3} there is a coloring $%
g\in I_{n}$ such that $d(f,g)<\epsilon /2$, so that $U(g,\epsilon
/2)\subseteq B_{n}(t)$. The idea is to modify slightly the cube coloring $g$
so that the new coloring will still be close to $g$, but there will be no
cube in $[-n,n]^{d}$ possessing a fair splitting of size at most $t$ and
granularity at least $1/n$. Without loss of generality we may assume that
there are equally spaced cubes $C_{i_{1},\ldots ,i_{d}}$ for $i_{1},\ldots
,i_{d}\in \{1,2,3,\ldots ,N\}$ in $[-n,n]^{d}$ such that $1>6n^{2}/N$ each
cube is filled with a unique color in the cube coloring $g$. Let $\delta >0$
be a real number satisfying 
\begin{equation*}
\delta <\min \left\{ \frac{\sqrt[d]{\epsilon }}{2N},\frac{2n}{N^{2}}\right\}
.
\end{equation*}%
Choose a color (which we will call from now on "white"). Let $%
W_{i_{1},\ldots ,i_{d}}^{\prime }$ where $i_{1},\ldots ,i_{d}\in
\{1,2,\ldots ,N\}$ be a cube $[0,2\delta ]^{d}$ colored as follows: choose a
countable set 
\begin{equation*}
\{m_{i_{1},\ldots ,i_{d}}^{j}\}_{j=1,\ldots ,k;\;i_{1},\ldots ,i_{d}\in
\{1,2,3\ldots N\}}
\end{equation*}%
of real numbers linearly independent over $\mathbb{Q}$ such that $%
0<m_{i_{1},\ldots ,i_{d}}^{j}<(\delta /k)^{d}$. We color $W_{i_{1},\ldots
,i_{d}}^{\prime }$ white except for small cubes 
\begin{equation*}
V_{i_{1},\ldots ,i_{d}}^{\eta }=\left( \frac{2\eta -1}{k}\delta ,\ldots ,%
\frac{2\eta -1}{k}\delta \right) +\prod_{j=1}^{d}\left[ -\sqrt[d]{%
m_{i_{1},\ldots ,i_{d}}^{j}},\sqrt[d]{m_{i_{1},\ldots ,i_{d}}^{j}}\right]
\end{equation*}%
colored using color $\eta $ for $\eta =1,2,\ldots k$. Note that the $d$%
-dimensional Lebesgue measure of $V_{i_{1},\ldots ,i_{d}}^{\eta }$ is equal $%
2^{d}m_{i_{1},\ldots ,i_{d}}^{j}$. Hence measures of these cubes are
linearly independent over $\mathbb{Q}$.

Now modify the coloring $g$ to get a coloring $h$ outside $B_{n}(t)$. The
coloring $h$ is equal to $g$ outside $[-n,n]^{d}$. Inside $C_{i_{1},\ldots
,i_{d}}$ the coloring $h$ is equal $g$ except in 
\begin{equation*}
W_{i_{1},\ldots ,i_{d}}=\left( \left( i_{1}-\frac{1}{2}-\delta \right) \frac{%
2n}{N}-n,\ldots ,\left( i_{d}-\frac{1}{2}-\delta \right) \frac{2n}{N}%
-n\right) +W_{i_{1},\ldots ,i_{d}}^{\prime }
\end{equation*}%
where $h$ is defined by the coloring of $W_{i_{1},\ldots ,i_{d}}^{\prime }$.

Note that $d(g,h)<\epsilon /2$ so that there exists a $d$-dimensional cube $%
C $ in $[-n,n]^{d}$ with granularity at least $1/n$ such that there is a
fair splitting of size at least $t$. The fair splitting divides $C$ into at
most $(t+1)^{d}$ cubes hence we obtain a $d$-dimensional cell complex in $%
[-n,n]^{d}$ (which we will also denote by $C$). Let us denote by $A$ the
measure of $C_{i_{1},\ldots ,i_{d}}\setminus W_{i_{1},\ldots ,i_{d}}$ (note
that $A$ does not depend on the set of indexes chosen and we may assume it
is linearly independent with the $m_{i_{1},\ldots ,i_{d}}^{j}$ chosen
before).

By the determinant of $C_{i_{1},\ldots ,i_{d}}$ in $C$ (denoted by $%
\det_{C}C_{i_{1},\ldots ,i_{d}}$) we mean the lowest dimension of cells $C$
that intersect $C_{i_{1},\ldots ,i_{d}}$ (there is only one cell reaching
the minimum -- denoted by $d_{C}(C_{i_{1},\ldots ,i_{d}})$). If $%
C_{i_{1},\ldots ,i_{d}}$ lays outside $C$ we set $\det_{C}C_{i_{1},\ldots
,i_{d}}=d$. Note that cells of $C$ divide each cube $C_{i_{1},\ldots ,i_{d}}$
into $2^{\limfunc{codim}(\det_{C}C_{i_{1},\ldots ,i_{d}})}$ cubes of
measures 
\begin{equation*}
\alpha _{1}(d_{C}(C_{i_{1},\ldots ,i_{d}})),\alpha
_{2}(d_{C}(C_{i_{1},\ldots ,i_{d}})),\ldots ,\alpha _{2^{\limfunc{codim}%
(\det_{C}C_{i_{1},\ldots ,i_{d}})}}(d_{C}(C_{i_{1},\ldots ,i_{d}}))
\end{equation*}%
and their sum is equal to $A$. In fact (up to indexing) $\alpha
_{i}(d_{C}(C_{i_{1},\ldots ,i_{d}}))$ does not depend on $%
d_{C}(C_{i_{1},\ldots ,i_{d}})$ but on the $\det_{C}(C_{i_{1},\ldots
,i_{d}}) $-dimensional subspace of $\mathbb{R}^{d}$ spanned by it. The
subspace can be identified by a suitable choice of $\limfunc{codim}(\det
C_{i_{1},\ldots ,i_{d}})$ slices (or ends) of $C$ on some of the dimensions.
Hence we get that $\alpha _{i}(d_{C}(C_{i_{1},\ldots ,i_{d}}))=\alpha
_{i}(t_{1},\ldots ,t_{s})$ for $t_{1},\ldots ,t_{s}\in \{0,1,2,\ldots ,t+1\}$
and $s=0,1,2,\ldots ,d$. Of course $\alpha _{1}(\emptyset )=A$.

Note that the dimension of the space spanned by $\alpha _{i}(t_{1},\ldots
,t_{s})$ where $s>0$ is no greater than $D(d)$. Now note that all the
vertices of $C$ are colored at most by $(t+2)^{d}$ colors. Moreover cells of
dimensions $d-1$ of $C$ intersect at most one of the cubes $V_{i_{1},\ldots
,i_{d}}^{\eta }\subseteq C_{i_{1},\ldots ,i_{d}}$ and two such cell
intersect the cubes of the same color if they span the same subspace of $%
\mathbb{R}^{d}$. Since there are at most $d(t+2)$ different subspaces of $%
\mathbb{R}^{d}$ obtained in such a way then if $C$ intersects one of $%
V_{i_{1},\ldots ,i_{d}}^{\eta }\subseteq C_{i_{1},\ldots ,i_{d}}$ then $d-1$
of $C$ also does and it has one of $d(t+2)$ colors.

Summing up, let us consider a color $c$ different from white and the $%
(t+2)^{d}+d(t+2)$ colors mentioned before. Since our splitting is fair, $d$%
-dimensional cells colored partially by $c$ can be divided into two families
having equal measure of $c$. Hence the measure satisfies equality of the
form:

\begin{multline*}
T(0)A+\sum \epsilon (0)_{i_{1},\ldots ,i_{d}}^{j}2^{d}m_{i_{1},\ldots
,i_{d}}^{j}+\sum S(0)_{t_{1},\ldots ,t_{s}}^{i}\alpha _{i}(t_{1},\ldots
,t_{s}) \\
-T(0)A-\sum \epsilon (0)_{i_{1},\ldots ,i_{d}}^{j}2^{d}m_{i_{1},\ldots
,i_{d}}^{j}-\sum S(0)_{t_{1},\ldots ,t_{s}}^{i}\alpha _{i}(t_{1},\ldots
,t_{s})=0
\end{multline*}%
where $T(0),T(1)\in \mathbb{N}$, $\epsilon (0)_{i_{1},\ldots
,i_{d}}^{j},\epsilon (1)_{i_{1},\ldots ,i_{d}}^{j}\in \{0,1\}$, $%
S(0),S(1)\in \mathbb{N}$, and not all $\epsilon (0)_{i_{1},\ldots ,i_{d}}^{j}
$ are equal to $0$. Note that for each color the numbers 
\begin{equation*}
T(0)A+\sum \epsilon (0)_{i_{1},\ldots ,i_{d}}^{j}2^{d}m_{i_{1},\ldots
,i_{d}}^{j}-T(0)A-\sum \epsilon (0)_{i_{1},\ldots
,i_{d}}^{j}2^{d}m_{i_{1},\ldots ,i_{d}}^{j}
\end{equation*}%
are independent over $\mathbb{Q}$. On the other hand, they can be generated
over $\mathbb{Q}$ by $\alpha _{i}(t_{1},\ldots ,t_{s})$ so they lie in $D(d)$%
-dimensional space. Since the number of remaining colors is greater than $%
D(d)$, we get a contradiction that ends the proof.
\end{proof}

\section{Open problems}

Let $f(t,d)$ denote the minimum number of colors needed for a coloring of $%
\mathbb{R}^{d}$ such that no cube has a fair splitting of size at most $t$.
Our main result asserts that $f(t,d)\leq
(t+4)^{d}-(t+3)^{d}+(t+2)^{d}-2^{d}+d(t+2)+3$. We expect naturally that this
bound is far from optimal, as even for $t=d=1$ it gives worst result than
that obtained in \cite{AlonGLM}.

\begin{problem}
Is it true that $f(t,d)\leq t+O(d)$?
\end{problem}

Turning into discrete case we get the following generalizations of the
problem of Erd\H{o}s. Let $g(t,d)$ denote the least number of colors needed
for a coloring of $\mathbb{Z}^{d}$ such that no cube has a fair splitting
using at most $t$ axes aligned cuts in total. So, by the result of Ker\"{a}%
nen we know that $g(1,1)=4$. Curiously we do not even know if $g(t,1)$ is
finite for every $t\geq 2$.

\begin{problem}
Determine $g(1,2)$ and $g(2,1)$.
\end{problem}

Finally let us formulate a natural discrete version of multidimensional
necklace splitting problem in the spirit of the theorem of de Longueville
and \v{Z}ivaljevi\'{c}. By a $d$\emph{-dimensional necklace} we mean a
discrete cube in $\mathbb{Z}^{d}$, that is a $d$-fold Cartesian product of
the set $\{1,2,\ldots ,n\}$ with itself. We assume that the necklace is
colored so that each color appears an even number of times. As before, in
the fair splitting problem we allow only axes aligned cuts.

\begin{problem}
What is the least number of axes aligned cuts needed for a fair splitting of 
$k$-colored $d$-dimensional necklace?
\end{problem}

Let $h(k,d)$ denote the number we asked for in the problem. It is not hard
to see that $h(k,d)\leq (2d-1)k$, as noticed by Laso\'{n} (personal
communication). For instance, if $d=2$ we may string the necklace as shown
in Fig. 1, and then apply one dimensional theorem. We get at most $k$ places
to cut the stringed necklace. However, to separate the resulting pieces
accordingly we need to use zigzags consisting of three line segments---one
horizontal, two vertical. Therefore for each cutting place of the string we
may need three orthogonal plane cuts. This proves the bound $h(k,2)\leq 3k$.
The argument for higher dimensions is analogous.

\begin{center}
\begin{equation*}
\FRAME{itbpF}{1.8844in}{2.1015in}{0in}{}{}{Figure}{\special{language
"Scientific Word";type "GRAPHIC";maintain-aspect-ratio TRUE;display
"USEDEF";valid_file "T";width 1.8844in;height 2.1015in;depth
0in;original-width 4.6804in;original-height 5.2226in;cropleft "0";croptop
"1";cropright "1";cropbottom "0";tempfilename
'MA4SCR01.bmp';tempfile-properties "XPR";}}
\end{equation*}

Fig. 1
\end{center}

The following construction due to Petecki (personal communication) shows
that the upper bound for $h(k,2)$ is close to the truth. Consider the set of
red points depicted in Fig. 2. It can be checked that any fair splitting of
this set must use at least three lines. So, taking $k-1$ copies of this set,
each in different color, and lying far one from another (so that there is no
vertical or horizonatal line crossing any two of the copies) one gets that $%
h(k,2)\geq 3(k-1)+1$. 
\begin{equation*}
\FRAME{itbpF}{1.8412in}{1.6155in}{0in}{}{}{Figure}{\special{language
"Scientific Word";type "GRAPHIC";maintain-aspect-ratio TRUE;display
"USEDEF";valid_file "T";width 1.8412in;height 1.6155in;depth
0in;original-width 4.056in;original-height 3.5552in;cropleft "0";croptop
"1";cropright "1";cropbottom "0";tempfilename
'MA54JR03.bmp';tempfile-properties "XPR";}}
\end{equation*}

\begin{center}
Fig. 2
\end{center}

\begin{acknowledgement}
Jaros\l aw Grytczuk acknowledges a support from Polish Ministry of Science
and Higher Education Grant (MNiSW) (N N206 257035).
\end{acknowledgement}

\end{document}